\documentclass[11pt,a4paper]{amsart}

\usepackage{amsmath,amssymb,amsthm,geometry,accents,csquotes,graphicx,tikz-cd,yhmath,newtxtext,xpatch}  
\usepackage[pdftex,hidelinks]{hyperref} % Makes references and citations into links
\usepackage[UKenglish]{babel}
\usepackage[T1]{fontenc}
 
\newtheorem{theorem}{Theorem}[section]
\newtheorem{lemma}[theorem]{Lemma}
\newtheorem{proposition}[theorem]{Proposition}

\newtheorem{innercustomthm}{Theorem}
\newenvironment{restate}[1]
{\renewcommand\theinnercustomthm{#1}\innercustomthm}
{\endinnercustomthm}

\theoremstyle{definition}
\newtheorem{remark}[theorem]{Remark}

\theoremstyle{definition}
\newtheorem{definition}[theorem]{Definition}

\DeclareMathOperator{\diam}{\mathsf{diam}}
\DeclareMathOperator{\dist}{\mathsf{dist}}

\newcommand{\calp}{\mathcal{P}}

\newcommand{\calk}{\mathcal{K}}

\makeatletter
\DeclareRobustCommand\bigop[2][1]{%
  \mathop{\vphantom{\sum}\mathpalette\bigop@{{#2}{#1}}}\slimits@
}
\newcommand{\bigop@}[2]{\bigop@@#1#2}
\newcommand{\bigop@@}[3]{%
  \vcenter{%
    \sbox\z@{$#1\sum$}%
    \hbox{\resizebox{#3\dimexpr\ifx#1\displaystyle.9\fi\dimexpr\ht\z@+\dp\z@}{!}{$\m@th#2$}}%
  }%
}
\makeatother

\begin{document}

\title{Morse and stable subgroups via the coset intersection complex}

\author[]{Tomohiro Fukaya}
\address{Department of Mathematical Sciences, Tokyo Metropolitan University, Tokyo, Japan}
\email{tomohirofukaya@tmu.ac.jp}

\author[]{Haoyang He}
\address{Department of Mathematics and Statistics, Memorial University of Newfoundland, St. John's, NL, Canada}
\email{haoyangh@mun.ca}

\author[]{Eduardo Martínez-Pedroza}
\address{Department of Mathematics and Statistics, Memorial University of Newfoundland, St. John's, NL, Canada}
\email{emartinezped@mun.ca}

\author[]{Takumi Matsuka}
\address{National Institute of Technology, Numazu College, Shizuoka, Japan}
\email{takumi.matsuka1@gmail.com}

\date{\today}

\begin{abstract}
    In this note, we study the equivalence of Morse and stable subgroups in the framework of the coset intersection complex. Under certain conditions on a coset intersection complex of a group, we prove that infinite-index Morse subgroups are stable. 
\end{abstract}

\maketitle

\section{Introduction}

The notion of stable subgroups was introduced by Durham\textendash Taylor~\cite{MR3426695}. The notion of Morse (or strongly quasiconvex) subgroups was introduced independently by Tran~\cite{MR3956891} and Genevois~\cite{MR4057355}. 

\begin{definition}
    Let $G$ be a finitely generated group and $H$ a subgroup.
    \begin{enumerate}
        \item The subgroup $H$ is \emph{stable} if $H$ is finitely generated undistorted subgroup such that for some (equivalently, every) finite generating set $S$ of $G$ and for every $k\geq1$ and $c\geq0$, there is $L = L(S,k,c)$  such that any two $(k,c)$-quasigeodesics in $G$ with the same endpoints on $H$ are at Hausdorff distance at most $L$. 
        \item The subgroup $H$ is \emph{Morse} or \emph{strongly quasiconvex} if for some (equivalently, every) finite generating set $S$ of $G$ and for every $k\geq1$  and $c\geq0$, there is $m=m(S,k,c)$  such that any $(k,c)$-quasigeodesic in $G$ with endpoints on $H$ is contained in the $M$-neighbourhood of $H$. 
    \end{enumerate}
\end{definition}

The two notions are equivalent if $G$ is hyperbolic, but they are not equivalent in general. For example, right-angled Coxeter groups contain subgroups that are Morse but not stable~\cite[Section 7]{MR3956891}, and non-hyperbolic peripheral subgroups of relatively hyperbolic groups are also Morse but not stable. It is, therefore, natural to ask when these two notions are equivalent. The following characterisation is due to Tran.

\begin{theorem}[{\cite[Theorem 4.8]{MR3956891}}]
    \label{thm_4.8tran19}
    Let $G$ be a finitely generated group. Let $H$ be an infinite subgroup of $G$. Then $H$ is stable if and only if $H$ is Morse and hyperbolic. 
\end{theorem} 

The two notions have shown to be equivalent for particular classes of groups, as the following results in the literature illustrate. Note that (2) is a special case of (3) in Theorem~\ref{thm_stabMorseMCGRAAGgpr}. 

\begin{theorem}
    \label{thm_stabMorseMCGRAAGgpr}
    If $H$ is a non-trivial infinite-index Morse subgroup of one of the following groups, then $H$ is stable.
    \begin{enumerate}
        \item~\cite[Corollary B]{MR3978121} A mapping class group of an oriented, connected, finite-type surface with complexity at least two. 
        \item~\cite[Theorem 8.19]{MR3956891} A right-angled Artin group whose defining graph is finite. 
        \item~\cite[Theorem 1.5]{balasubramanya2026stablesubgroupsgraphproducts} A graph product whose defining graph is finite, connected with at least two vertices, and whose vertex groups are infinite.
    \end{enumerate}
\end{theorem}

Some of these results are particular instances of a more general statement, which this note presents. We study the equivalence of Morse and stable in the framework of the coset intersection complex, introduced by Abbott\textendash Martínez-Pedroza~\cite{MR5034311}. 

\begin{definition}
    A \emph{group pair} $(G,\calp)$ consists of a finitely generated group $G$ and a finite collection $\calp$ of infinite subgroups. The \emph{coset intersection complex} $\calk(G,\calp)$ of a group pair is the simplicial complex whose vertices are left cosets $gP$ where $g\in G$ and $P\in\calp$, and whose $r$-simplices are collections of $r+1$ cosets $\{g_0P_0,\cdots,g_rP_r\}$ such that the intersections $\bigcap_{i=0}^rg_iP_ig_i^{-1}$ are infinite. 
\end{definition}

The main result of this note is the following. 

\begin{restate}{\ref{thm_Mstablevf}}
    Let $(G,\calp)$ be a group pair such that the following four conditions hold:
    \begin{enumerate}
        \item Each $P\in\mathcal{P}$ is undistorted in $G$.
        \item $\mathcal{P}$ is quasi-syllabic.
        \item For each $P\in\mathcal{P}$, if $Q\leqslant P$ is Morse in $P$, then $Q$ is finite or finite-index in $P$.
        \item The coset intersection complex $\mathcal{K}(G,\mathcal{P})$ is connected, hyperbolic and has $G$-cocompact one-skeleton.
    \end{enumerate}
    If $H$ is a non-trivial infinite-index Morse subgroup of $G$, then $H$ is a stable subgroup. Moreover $H$ is virtually free. 
\end{restate}

See Definition~\ref{def_quasisyl} for the notion of quasi-syllabic. Our proof is inspired by the work of Tran~\cite{MR3956891}. We remark that a similar strategy in the context of graph products is adapted by Balasubramanya\textendash Chesser\textendash Kerr\textendash Mangahas\textendash Trin~\cite[Section 6]{balasubramanya2026stablesubgroupsgraphproducts}. On the way to proving Theorem~\ref{thm_Mstablevf}, we obtain the following theorem.

\begin{restate}{\ref{thm_gaot}}
    Let $G$ be a finitely generated group acting cocompactly on a simplicial tree $T$. Suppose that 
    \begin{enumerate}
        \item Every edge stabiliser is infinite.
        \item Every vertex stabiliser is undistorted in $G$.
        \item Every vertex stabiliser has only finite or finite-index Morse subgroups.
    \end{enumerate}
    If $H$ is an infinite-index Morse subgroup of $G$, then $H$ is virtually free and stable.
\end{restate}

We remark that, in the first version of this article, the result above requires the action to be acylindrical, so that the tree is quasi-isometric to a coset intersection complex. This condition is not required in the new proof. 

\subsection*{Acknowledgements}

Many thanks to Anthony Genevois for continuous feedback and corrections on our work, in particular, for introducing us to the notion of quasi-syllabic and explaining us how to use it in our context. We also thank Monika Kudlinska for a correction in the first version of this note. TF is supported by JSPS KAKENHI Grant number 24K06741. HH is partially supported by the School of Graduate Studies at Memorial University of Newfoundland. EMP acknowledges funding by the Natural Sciences and Engineering Research Council (NSERC) of Canada. 

\section{Proof of Theorem~\ref{thm_Mstablevf}}
\label{sec_proofmain}

In this section we prove the following theorem. 

\begin{theorem}
    \label{thm_Mstablevf}
    Let $(G,\mathcal{P})$ be a group pair such that the following four conditions hold:
    \begin{enumerate}
        \item Each $P\in\mathcal{P}$ is undistorted in $G$.
        \item $\mathcal{P}$ is quasi-syllabic.
        \item For each $P\in\mathcal{P}$, if $Q\leqslant P$ is Morse in $P$, then $Q$ is finite or finite-index in $P$.
        \item The coset intersection complex $\mathcal{K}(G,\mathcal{P})$ is connected, hyperbolic and has $G$-cocompact one-skeleton. 
    \end{enumerate}
    If $H$ is an infinite-index Morse subgroup of $G$, then $H$ is a stable subgroup.   
\end{theorem}

We start with the following proposition. For a group pair $(G,\calp)$, we denote $G/\calp=\{gP\mid g\in G,P\in\calp\}$.

\begin{proposition}
    \label{prop_HgPgfin}
    Let $(G,\mathcal{P})$ be a group pair such that the following three conditions hold:
    \begin{enumerate}
        \item Each $P\in\mathcal{P}$ is undistorted in $G$.
        \item For each $P\in\mathcal{P}$, if $Q\leqslant P$ is Morse in $P$, then $Q$ is finite or finite-index in $P$.
        \item The coset intersection complex $\mathcal{K}(G,\mathcal{P})$ is connected. 
    \end{enumerate}
    Let $H$ be an infinite-index Morse subgroup of $G$. Then for every $gP\in G/\calp$,  $H\cap gPg^{-1}$ is finite. 
\end{proposition}

The proof of this proposition requires the following three lemmata. 

\begin{lemma}
    For subgroups, the properties of being undistorted and being Morse are preserved under conjugation by any element of the ambient group. 
\end{lemma}

\begin{proof}
   Let $H$ be an undistorted (resp. Morse) subgroup of $G$ and $g\in G$. Since multiplication on the left by $g$ defines an isometry $G\to G$, the subspace $gH$ is undistorted (resp. Morse) in $G$. On the other hand, multiplication on the right by $g$ defines a quasi-isometry $G\to G$. Since quasi-isometries maps (quasi-)geodesics to (quasi-)geodesics, $gHg^{-1}$ is undistorted (resp. Morse).
\end{proof}

The above lemma implies that under the assumptions of Proposition~\ref{prop_HgPgfin}, any subgroup $gPg^{-1}$, with $gP\in G/\calp$, is undistorted and its Morse subgroups are either finite or finite-index.   

\begin{lemma}[{\cite[Proposition 4.11]{MR3956891}}]
    \label{lem_prop4.11tran}
    Let $G$ be a finitely generated group. Let $P$ be an undistorted subgroup of $G$. If $A$ is a Morse subgroup of $G$, then $A\cap P$ is a Morse subgroup of $P$. In particular, $A\cap P$ is finitely generated and undistorted in $P$.
\end{lemma}

\begin{lemma}
    \label{lem_ficonj}
    Let $(G,\mathcal{P})$ be a group pair with the same assumptions of Proposition~\ref{prop_HgPgfin}. Let $P\in\mathcal{P}$. Let $H$ be a Morse subgroup of $G$. 
    \begin{enumerate}
        \item[(i)] If $H\cap gPg^{-1}$ is infinite, then $H\cap gPg^{-1}$ is finite-index in $gPg^{-1}$.
        \item[(ii)] If there is $gP\in G/\mathcal{P}$ such that $H\cap gPg^{-1}$ is infinite, then for every $hQ\in G/\mathcal{P}$, $H\cap hQh^{-1}$ is infinite and finite-index in $hQh^{-1}$.
    \end{enumerate}
\end{lemma}

The proof of this lemma, and the proof of Proposition~\ref{prop_HgPgfin}, follow similar lines of the proof of~\cite[Proposition 8.18]{MR3956891}.

\begin{proof}
    To prove (i), since $H$ is Morse in $G$ and $gPg^{-1}$ is undistorted in $G$, the subgroup $H\cap gPg^{-1}$ is Morse in $gPg^{-1}$ by Lemma~\ref{lem_prop4.11tran}. By hypothesis (2), this implies that $H\cap gPg^{-1}$ is finite-index in $gPg^{-1}$. 

    To show (ii), since $\mathcal{K}(G,\mathcal{P})$ is connected, there is a path with vertices \[gP=g_0P_0,g_1P_1,\cdots,g_{n-1}P_{n-1},g_nP_n=hQ\] in $\mathcal{K}(G,\mathcal{P})$. By definition, $g_iP_ig_i^{-1}\cap g_{i+1}P_{i+1}g_{i+1}^{-1}$ is infinite for every $0\le i\le n-1$. Since $H\cap g_0P_0g_0^{-1}$ is infinite, $H\cap g_0P_0g_0^{-1}$ is finite-index in $g_0P_0g_0^{-1}$ by (i). It follows that $H\cap g_0P_0g_0^{-1}\cap g_1P_1g_1^{-1}$ is finite-index in $g_0P_0g_0^{-1}\cap g_1P_1g_1^{-1}$, and $H\cap g_0P_0g_0^{-1}\cap g_1P_1g_1^{-1}$ is infinite since $g_0P_0g_0^{-1}\cap g_1P_1g_1^{-1}$ is infinite. In particular, $H\cap g_1P_1g_1^{-1}$ is infinite, so part (i) implies that $H\cap g_1P_1g_1^{-1}$ is finite-index in $g_1P_1g_1^{-1}$. By induction, we have that $H\cap hQh^{-1}$ is infinite and finite-index in $hQh^{-1}$. 
\end{proof}

Before proving Proposition~\ref{prop_HgPgfin} we recall the following definition. We say that a subgroup $H$ of a group $G$ has \emph{height at most $n$} if any $(n+1)$ conjugates $g_0Hg_0^{-1},\cdots,g_nHg_n^{-1}$ with $g_0H,\cdots,g_nH$ distinct have finite intersection. 

\begin{proof}[Proof of Proposition~\ref{prop_HgPgfin}]
    Assume for a contradiction that there is $g_0P_0\in G/\mathcal{P}$ such that $H\cap g_0P_0g_0^{-1}$ is infinite. Since $H$ is Morse in $G$, it has finite height~\cite[Theorem 4.15]{MR3956891}, say $n$. Since $H$ is infinite-index in $G$, take $n+1$ distinct cosets $g_1H,\cdots,g_{n+1}H$. Since $H\cap g_0P_0g_0^{-1}$ is infinite, Lemma~\ref{lem_ficonj}(ii) implies that  $H\cap g_i^{-1}g_0P_0g_0^{-1}g_i$ is finite-index in $g_i^{-1}g_0P_0g_0^{-1}g_i$. In particular, we have that  $g_iHg_i^{-1}\cap g_0P_0g_0^{-1}$ is finite-index in $g_0P_0g_0^{-1}$ for every $1\le i\le n+1$. Therefore $\left(\bigcap_{i=1}^{n+1}g_iHg_i^{-1}\right)\cap g_0P_0g_0^{-1}$ is finite index in $g_0P_0g_0^{-1}$. In particular, $(\bigcap_{i=1}^{n+1}g_iHg_i^{-1})$ is infinite since $g_0P_0g_0^{-1}$ is infinite, contradicting that $H$ has height $n$. 
\end{proof}

Given a group pair $(G,\calp)$, denote $\Gamma(G)$, $\widehat\Gamma(G,\mathcal{P})$ and $\mathcal{K}(G,\mathcal{P})$ as the Cayley graph, the coned-off Cayley graph, and the coset intersection complex,  respectively.   Let $\widetriangle\Gamma=\widetriangle\Gamma(G,\mathcal{P})$ be the complex obtained from $\widehat\Gamma$ by adding all edges of $\mathcal{K}$, i.e. $\widetriangle\Gamma =\widehat\Gamma\cup\mathcal{K}$. Since $\widetriangle\Gamma$ is obtained from $\mathcal{K}$ by adding a finite number of $G$-orbits of edges, we have the following statement. 

\begin{lemma} 
	\begin{enumerate}
		\item[]
		\item If $\mathcal{K}$ is connected, then the inclusion $\mathcal{K}\hookrightarrow \widetriangle\Gamma$ is a $G$-equivariant quasi-isometry.
		\item  If $\mathcal{K}$ has cocompact one-skeleton, then $\widehat \Gamma \hookrightarrow \widetriangle\Gamma$ is a $G$-equivariant quasi-isometry.  
		\item If $\mathcal{K}$ is connected and has $G$-cocompact one-skeleton, then $\widehat\Gamma$ and $\mathcal{K}$ are quasi-isometric.
	\end{enumerate}
\end{lemma}

A path in a graph is a sequence of vertices $v_0,v_1,\ldots , v_\ell$ such that $v_i$ and $v_{i+1}$ are adjacent. An extension of a path $\widehat\gamma$ in $\widehat\Gamma$ to $\Gamma$ is a path in $\Gamma$ obtained by replacing each cone vertex $\widehat{gP}$ in $\widehat\gamma$ with a path in $\Gamma$ containing only vertices in the coset $gP$.

\begin{definition}
    \label{def_quasisyl}
    The collection of subgroups $\mathcal{P}$ is \emph{quasi-syllabic} if there exists $A>0$ and $B\geq0$ such that for any two elements of $G$ there is an $(A,B)$-quasigeodesic in $\widehat\Gamma$ that extends to an $(A,B)$-quasigeodesic in $\Gamma$.  
\end{definition}

\begin{proposition}
    \label{lem:Anthony}
    Suppose $Q\cap gPg^{-1}$ is a finite subgroup for every $gP\in G/\mathcal{P}$. If $\mathcal{P}$ is quasi-syllabic and $Q$ is a Morse subgroup, then $Q\hookrightarrow\widehat\Gamma$ is a quasi-isometric embedding.    
\end{proposition}

We need the following lemma to prove the proposition above. 

\begin{lemma}
	If $Q\cap gPg^{-1}$ is a finite subgroup for every $gP\in G/\mathcal{P}$, then $Q$ has uniformly bounded coarse intersections with $\mathcal{P}$, i.e. for every $L\geq0$ there is $K\geq0$ such that $\mathsf{diam}_\Gamma(Q^{+L}\cap gP^{+L})\leq K$ for every $gP\in G/\mathcal{P}$.
\end{lemma}

For a subset $Q\subseteq G$ and $L>0$, $Q^{+L}$ denotes the $L$-neighbourhood of $Q$ in $G$. 

\begin{proof}
	Let $L\geq 0$. For any coset  $gP\in G/\mathcal{P}$, there is a constant $M(gP,Q, L)$ such that \[ Q^{+L} \cap gP^{+L}\subseteq(Q\cap gPg^{-1})^{+M(gP,Q, L)},\] see for example~\cite[Lemma 4.4]{MR5034311} and references there in. In particular, \[\diam_\Gamma(Q^{+L} \cap gP^{+L})  \leq  \diam_\Gamma\left((Q\cap gPg^{-1})^{+M(gP,Q, L)}\right)<\infty,\] since $Q\cap gHg^{-1}$ is a finite group and $\Gamma$ is a locally finite graph. Let \[K=\max\{ \diam_\Gamma\left((Q\cap xPx^{-1})^{+M(xP,Q, L)}\right) \colon x\in G,\ |x|\leq 2L,\ P\in \mathcal{P}  \} < \infty\] and observe that $K$ is a finite number since $\mathcal{P}$ is a finite collection and there are only finitely many $x\in G$ with $|x|<2L$.
	
	Now we argue that $K$ satisfies the conclusion of the lemma. Suppose that $gP\in G/\mathcal{P}$ is arbitrary and $Q^{+L}\cap gP^{+L}$ is non-empty. Then there is an element $q\in Q$ such that $\dist_\Gamma(q, gP)\leq 2L$. In particular $gP=qxP$ where $x\in G$ and $|x|\leq 2L$. It follows that $\diam_\Gamma(Q^{+L}\cap gP^{+L}) = \diam_\Gamma(Q^{+L}\cap xP^{+L}) \leq K$.  
\end{proof}
 
 \begin{proof}[Proof of Proposition~\ref{lem:Anthony}]
 	Since both $Q$ and $\widehat\Gamma$ are quasi-geodesic spaces, it is enough to show that $Q \hookrightarrow \widehat\Gamma$ is a coarse embedding.
 	
 	Let $\dist$ denote the word metric on $Q$ with respect to some finte generating set. Since $Q\leq G$, we can assume that $\dist_G \leq \dist$ on $Q$. Since $\widehat\dist \leq \dist_G$, we have that $\widehat\dist \leq \dist$. 
 	
 	Suppose that $Q$ does not coarsely embed into $\widehat\Gamma$. Then there exists a constant $D>0$ and a sequence $q_n$ of elements of $Q$ such that $\widehat\dist(e,q_n)\leq D$ and $\dist(e,q_n)\to\infty$ for every $n\geq1$.  Since $\mathcal{H}$ is quasi-syllabic, there are constants $A>0$ and $B\geq0$ and $(A,B)$-quasigeodesics from $e$ to $q_n$ in $\widehat\Gamma$ which extend to $(A,B)$-quasigeodesics $\alpha_n$ in $\Gamma$.
 	
 	On the other hand, that $\dist(e,q_n)\to\infty$ implies that for any $K>0$ there is $n$ such that $\alpha_n$ contains a subpath of length $K$ with vertices in some left coset $gH$. Since $Q$ is Morse this subpath should be contained in an $L$-neighborhood of $Q$ where $L=L(Q,A,B)$, and that means $\diam_\Gamma(Q^{+L}\cap gH^{+L})\geq \frac{K}{A}-B$. As a consequence, $Q$ does not have  uniformly bounded coarse intersections with $\mathcal{P}$, a contradiction. 
 \end{proof}

We are now ready to prove the main theorem. 

\begin{proof}[Proof of Theorem~\ref{thm_Mstablevf}]
	By Propositions~\ref{prop_HgPgfin} and \ref{lem:Anthony}, the assumption that $Q$ is Morse and $\mathcal{P}$ is quasi-syllabic implies that the orbit inclusion $Q\hookrightarrow \widehat\Gamma$ is a quasi-isometric embedding.  Since $\mathcal{K}(G,\mathcal{P})$ is connected and has $G$-cocompact one-skeleton, it is quasi-isometric to $\widehat\Gamma$. It follows that $\widehat\Gamma$ is hyperbolic and hence $Q$ as well.  Hence $Q$ is stable by Theorem~\ref{thm_4.8tran19}. 
\end{proof}

\section{Proof of Theorem~\ref{thm_gaot}}
\label{sec_proofgaot}

In this section we prove the following theorem.

\begin{theorem}
    \label{thm_gaot}
    Let $G$ be a finitely generated group acting cocompactly on a simplicial tree $T$. Suppose that 
    \begin{enumerate}
        \item Every edge stabiliser is infinite.
        \item Every vertex stabiliser is undistorted in $G$.
        \item Every vertex stabiliser has only finite or finite-index Morse subgroups.
    \end{enumerate}
    If $H$ is an infinite-index Morse subgroup of $G$, then $H$ is virtually free and stable.
\end{theorem}

\begin{proof}
    Let $\calp$ be a collection of stabilisers of representatives of $G$-action on the vertex set. The collection $G/\calp$ is in one-to-one $G$-equivariant correspondence with the vertex set of $T$. Assumption (1) implies that we have a $G$-equivariant embedding $T\to\calk(G,\calp)$, so $\calk(G,\calp)$ is connected. Therefore, by Proposition~\ref{prop_HgPgfin}, $H$ acts properly on $\calk(G,\calp)$, and it follows that the $H$-action on $T$ is proper. 

    Since that $H$ is finitely generated, so we can choose $v\in T$ and a finite generating set $\{s_1,\cdots,s_n\}$ of $H$. Let $K$ be the union of geodesics $[v,s_iv]$, where $1\le i\le n$. Let $Y$ be the union of $hK$ where $h\in H$. Then $Y$ is connected subspace of $T$ and it is quasi-isometric to $H$, so $Y$ is a tree and $G$ is virtually free. 
\end{proof}

\begin{remark}
    Given a finitely generated group $G$ acting cocompactly and acylindrically on a simplicial tree $T$ with infinite edge stabilisers, then one can show that the coset intersection complex $\calk(G,\calp)$ defined in the proof above is quasi-isometric to $T$.
\end{remark}

\bibliographystyle{alphaurl} 
\bibliography{xbib}

@article {MR3978121,
    AUTHOR = {Kim, Heejoung},
     TITLE = {Stable subgroups and {M}orse subgroups in mapping class
              groups},
   JOURNAL = {Internat. J. Algebra Comput.},
  FJOURNAL = {International Journal of Algebra and Computation},
    VOLUME = {29},
      YEAR = {2019},
    NUMBER = {5},
     PAGES = {893--903},
      ISSN = {0218-1967,1793-6500},
   MRCLASS = {20F65 (20F67 57M07)},
  MRNUMBER = {3978121},
MRREVIEWER = {Stefan\ Witzel},
       DOI = {10.1142/S0218196719500346},
}

@article {MR3956891,
    AUTHOR = {Tran, Hung Cong},
     TITLE = {On strongly quasiconvex subgroups},
   JOURNAL = {Geom. Topol.},
  FJOURNAL = {Geometry \& Topology},
    VOLUME = {23},
      YEAR = {2019},
    NUMBER = {3},
     PAGES = {1173--1235},
      ISSN = {1465-3060,1364-0380},
   MRCLASS = {20F67 (20F65)},
  MRNUMBER = {3956891},
MRREVIEWER = {Jack\ O.\ Button},
       DOI = {10.2140/gt.2019.23.1173},
}

@article {MR5034311,
    AUTHOR = {Abbott, Carolyn and Mart\'inez-Pedroza, Eduardo},
     TITLE = {The quasi-isometry invariance of the coset intersection
              complex},
   JOURNAL = {Algebr. Geom. Topol.},
  FJOURNAL = {Algebraic \& Geometric Topology},
    VOLUME = {26},
      YEAR = {2026},
    NUMBER = {2},
     PAGES = {659--698},
      ISSN = {1472-2747,1472-2739},
   MRCLASS = {20F65 (57M07)},
  MRNUMBER = {5034311},
       DOI = {10.2140/agt.2026.26.659},
}

@misc{balasubramanya2026stablesubgroupsgraphproducts,
      title={Stable subgroups of graph products}, 
      author={Sahana H Balasubramanya and Marissa Chesser and Alice Kerr and Johanna Mangahas and Marie Trin},
      year={2026},
      eprint={2511.11176},
      archivePrefix={arXiv},
      primaryClass={math.GR},
}

@article {MR3426695,
    AUTHOR = {Durham, Matthew Gentry and Taylor, Samuel J.},
     TITLE = {Convex cocompactness and stability in mapping class groups},
   JOURNAL = {Algebr. Geom. Topol.},
  FJOURNAL = {Algebraic \& Geometric Topology},
    VOLUME = {15},
      YEAR = {2015},
    NUMBER = {5},
     PAGES = {2839--2859},
      ISSN = {1472-2747,1472-2739},
   MRCLASS = {20F65 (30F60 57M07)},
  MRNUMBER = {3426695},
MRREVIEWER = {Sebastian\ Wolfgang\ Hensel},
       DOI = {10.2140/agt.2015.15.2839},
}

@article {MR4057355,
    AUTHOR = {Genevois, Anthony},
     TITLE = {Hyperbolicities in {${\rm CAT}(0)$} cube complexes},
   JOURNAL = {Enseign. Math.},
  FJOURNAL = {L'Enseignement Math\'ematique},
    VOLUME = {65},
      YEAR = {2019},
    NUMBER = {1-2},
     PAGES = {33--100},
      ISSN = {0013-8584,2309-4672},
   MRCLASS = {20F65 (20F67 53C23)},
  MRNUMBER = {4057355},
MRREVIEWER = {Thomas\ Haettel},
       DOI = {10.4171/lem/65-1/2-2},
}

\end{document}